\documentclass[12pt]{amsart}
\usepackage{amsmath,amssymb,amsthm,amscd}
\newtheorem{formula}{}[section]
\newtheorem{proposition}[formula]{Proposition}
\newtheorem{corollary}[formula]{Corollary}

\newtheorem{theorem}[formula]{Theorem}
\theoremstyle{definition}
\newtheorem{definition}[formula]{Definition}
\newtheorem{example}[formula]{Example}
\theoremstyle{remark}
\newtheorem*{remark}{Remark}
\newtheorem*{conjecture}{Conjecture}

\begin{document}
\title[The Variety of Lie algebras of maximal class]
{The Variety of Lie algebras of maximal class}
\author{Dmitry Millionshchikov}
\thanks{It is partially supported by the RFBR grants
08-01-00541 and 08-01-91855, by the President grant of the support
of Scientific schools 2185.2008.1 and the grant "The Universities
of Russia"}
\address{Department of Mechanics and Mathematics of the Moscow State University (Lomonosov), 119992 Moscow, Russia}
\email{million@higeom.math.msu.su}

\begin{abstract}
We present an explicit description of the affine variety $M_{\it
Fil}$ of Lie algebras of the maximal class (filiform Lie
algebras): the formulas of polynomial equations that determine
this variety are written. The affine variety  $M_{\it Fil}$ can be
considered as the base of the nilpotent versal deformation of
${\mathbb N}$-graded Lie algebra $\mathfrak{m}_0$.
\end{abstract}
\date{}
\maketitle
\section*{Introduction}
The methods of the Lie algebras deformation theory are used in the
study of nilpotent lie algebras for many years. One knows that
there are only a finite number of pairwise non isomorphic Lie
algebras in dimensions less or equal six, but already in the
dimension  $7$ an one-parameter family of isomorphism classes of
such algebras appears for the first time. In the higher dimensions
the {\it varieties of nilpotent Lie algebras} are considered. More
often it means the variety of Lie algebra laws, defined on the
$n$-dimensional vector space ${\mathbb K}^n$ with a fixed basis
$e_1,\dots,e_n$. As a result one gets an affine variety in
${\mathbb K}^{n^3}$ (the coordinates in ${\mathbb K}^{n^3}$ are
the components of the tensor $c_{ij}^k$, $[e_i,e_j]=c_{ij}^ke_k$,
and polynomial equations on $c_{ij}^k$ are obtained from the
Jacobi identity as well as from from the nilpotency condition).
The generic points of the variety $M^n$ are so called {\it
filiform} Lie algebras (the term that was introduced by Michelle
Vergne) -- nilpotent Lie algebras with the maximal nil-index
$n{-}1$ for a given dimension $n$. Vergne have shown, that an
arbitrary filiform Lie algebra is isomorphic to some deformation
$[,]+\Psi$ of the graded filiform Lie algebra ${\mathfrak m}_0(n)$
with the bracket $[,]$, when the integrability condition for the
cocycle $\Psi \in H^2_+({\mathfrak m}_0(n),{\mathfrak m}_0(n))$
(i.e. the Jacobi identity for $[,]+\Psi$) is equivalent to the
vanishing of its Nijenhuis-Richardson square $\left[\Psi, \Psi
\right]=0$. The algebra $\mathfrak{m}_0(n)$ is defined by its
basis $e_1,\dots,e_n$ and nontrivial commutator relations:
$[e_1,e_i]=e_{i+1}, i=2,\dots,n{-}1$. Vergne have calculated the
dimensions of the spaces $H^2_+({\mathfrak m}_0(n),{\mathfrak
m}_0(n))$. Later her approach has become the cornerstone of the
whole number of researches (classifications of filiform Lie
algebras of small dimensions, the description of the irreducible
components of the variety $M^n$, and etc.) of the whole number of
authors:  Khakimdjanov, Goze, Ancochea-Bermudez and others. after
decomposing vector $\Psi=\sum x_{i,r}\Psi_{i,r}$ on the basis
$\Psi_{i,r}$ of the space $H^2_+({\mathfrak m}_0(n),{\mathfrak
m}_0(n))$ (the basis that was introduced by Khakimdjanov
\cite{Kh}) the integrability condition $\left[\Psi, \Psi
\right]=0$ is equivalent to some system of quadratic equations on
variables $x_{i,r}$. The Affine variety, defined by this system,
began to be called {\it the variety of filiform Lie algebras}. At
the same time, the study of these brackets up to an isomorphism
was carried out only in small dimendsions($\dim {\mathfrak g} \le
12$) and the answer, as a rule, was given in the form of the table
of the computer calculations with the aid of some package of the
symbolic calculations like Maple or Mathematica. In this case the
whole set of equations on unknowns $x_{i,r}$ for an arbitrary
dimension $n$ have been never written out. For instance in
\cite{EcMN} only an algorithm for finding of such equations was
proposed.

Beginning from the 50ths of the last century {\it $p$-groups of
maximal class} began actively to be studied in group theory:
groups of order $p^n$ and of index of the nilpotency $n{-}1$.
Later for the study of $p$-groups and pro-$p$-groups the notion of
index of the nilpotency have been substituted on a new invariant,
on the so-called {\it coclass} of the group. The application of
methods of Lie algebra theory with the study of $p$- and
pro-$p$-groups led to the fact, that  Lie algebras of maximal
class began to studied in finite characteristic. In some papers
they are called also {\it narrow} or {\it thin} Lie algebras (see
\cite{ShZ2}). Residually nilpotent Lie algebra $\mathfrak{g}$ is
called a Lie algebra of {\it maximal class} or of coclass $1$, if
$$\sum_{i\ge 1}(\dim
(C^i\mathfrak{g}/C^{i+1}\mathfrak{g})-1)=1,$$ where
$C^i\mathfrak{g}$ denotes $i$-th ideal of lower descending series
of $\mathfrak g$. If $\dim\mathfrak{g}< \infty$,  the this
condition is equivalent to the fact?, that $\mathfrak{g}$ is a
filiform Lie algebra.

The most elementary example of infinite dimensional Lie algebra of
maximal class is the algebra $\mathfrak{m}_0$
--- the direct limit of algebras $\mathfrak{m}_0(n)$.
It follows from Vergne's theorem \cite{V}, that an arbitrary
finitely generated Lie algebra of maximal class is isomorphic to
some special deformation of $\mathfrak{m}_0$. The space
 of such deformations has the structure of a direct limit
$M_{\it Fil}$ of affine varieties, imbedded in $\lim_n {\mathbb
K}^n$. $M_{\it Fil}$ is a base of nilpotent versal deformation of
the algebra ${\mathfrak m}_0$. The main result of the present
article is the Theorem \ref{main} with explicit formulas of
polynomial equations, that define $M_{\it Fil}$. The answer in
finite dimensional case is also provided. On has to remark that
the study of homogeneous deformations of $\mathfrak{m}_0$ was
started by Fialowski and Wagemann in \cite{FialWag}.

The paper is organized in the following way. In the first section
definitions and facts concerning Lie algebras of maximal class are
presented. Section 2 is devoted to the cohomology of ${\mathbb
N}$-graded Lie algebra ${\mathfrak m}_0$. In the third section we
describe the affine variety of Lie algebra s of maximal class in
the form of nilpotent versal deformation. Let us remark that all
considerations are valid for a field of zero characteristic. In
the section 4 we consider a variety of filiform Lie algebras of
finite dimension $n$. The number of unknowns of the system is
equal to the dimension of non negative part $\dim H^2_+({\mathfrak
m}_0((n), {\mathfrak m}_0(n))$ of the second cohomology with the
coefficients in the adjoint representation \cite{V}, and the
dimension $\dim H^3_+({\mathfrak m}_0(n), {\mathfrak m}_0(n))$
answers for the number of equations.

\section{Filiform Lie algebras and Lie algebras of maximal class}
The sequence of ideals of a Lie algebra$\mathfrak{g}$
$$C^1\mathfrak{g}=\mathfrak{g} \; \supset \;
C^2\mathfrak{g}=[\mathfrak{g},\mathfrak{g}] \; \supset \; \dots \;
\supset \; C^{k}\mathfrak{g}=[\mathfrak{g},C^{k-1}\mathfrak{g}] \;
\supset \; \dots$$ is called the lower descending central series
of the Lie algebra $\mathfrak{g}$.

A Lie algebra $\mathfrak{g}$ is called nilpotent, if there exists
$s$ such that
$$C^{s+1}\mathfrak{g}=[\mathfrak{g}, C^{s}\mathfrak{g}]=0,
\quad C^{s}\mathfrak{g} \: \ne 0.$$ The natural number $s$ is said
to be the nil-index of the nilpotent Lie algebra $\mathfrak{g}$.
\begin{proposition}
Let $\mathfrak{g}$ be a $n$-dimensional nilpotent Lie algebra.
Then we have the following estimate for its nil-index: $\quad s
\le n-1$.
\end{proposition}
\begin{definition}
A nilpotent $n$-dimensional Lie algebra $\mathfrak{g}$ is called
filiform if its nil-index is equal to $s=n-1$.
\end{definition}
\begin{example}
A Lie algebra $\mathfrak{m}_0(n)$, that is defined by its basis
$e_1, e_2, \dots, e_n$ with commutator relations:
$$ [e_1,e_i]=e_{i+1}, \; \forall \; 2\le i \le n{-}1,$$
is filiform.
\end{example}
Let $\mathfrak{g}$ be a Lie algebra. We will call the set $F$ of
subspaces
$$\mathfrak{g} \supset \ldots \supset F^{i}\mathfrak{g} \supset
F^{i{+}1}\mathfrak{g} \supset \ldots, \qquad i \in \mathbb Z,
$$
a decreasing filtration $F$ of the Lie algebra $\mathfrak{g}$, if
the set $F$ is compatible with the Lie algebra structure:
$$[F^{k}\mathfrak{g},F^{l}\mathfrak{g}]
\subset F^{k+l}\mathfrak{g}\qquad \forall k,l \in \mathbb Z.$$ Let
$\mathfrak{g}$ be a filtered Lie algebra. The graded Lie algebra
$${\rm gr}_F\mathfrak{g}=\bigoplus_{k=1} ({\rm
gr}_F\mathfrak{g})_k, \qquad ({\rm
gr}_F\mathfrak{g})_k=F^{k}\mathfrak{g}/F^{k{+}1}\mathfrak{g},$$ is
called the the associated graded Lie algebra ${\rm
gr}_F\mathfrak{g}$.

The ideals $C^{k}\mathfrak{g}$ of lower descending series form
decreasing filtration $C$ of a Lie algebra $\mathfrak{g}$
$$C^{1}\mathfrak{g}=\mathfrak{g} \supset C^{2}\mathfrak{g} \supset
\ldots \supset C^{k}\mathfrak{g} \supset \ldots, \qquad
[C^{k}\mathfrak{g},C^{l}\mathfrak{g}] \subset
C^{k+l}\mathfrak{g}.$$ Let us consider also the associated graded
Lie algebra${\rm gr}_C\mathfrak{g}$.
\begin{proposition}
Let $\mathfrak{g}$ be a filiform Lie algebra and  ${\rm
gr}_C\mathfrak{g}= \oplus_i ({\rm gr}_C\mathfrak{g})_i$ be the
corresponding associated (with respect to the canonical filtration
$C$) graded Lie algebra. Then
$$
\dim ({\rm gr}_C\mathfrak{g})_1=2, \quad \dim ({\rm
gr}_C\mathfrak{g})_2=\ldots= \dim ({\rm
gr}_C\mathfrak{g})_{n{-}1}=1.
$$
\end{proposition}
\begin{corollary}
Let $\mathfrak{g}$ be a filiform Lie algebra, then
$$\sum_{i=1}^{\dim{\mathfrak g}{-}1}(\dim (C^i\mathfrak{g}/C^{i+1}\mathfrak{g})-1)=1,$$
where only the first summand is non equal to zero.
\end{corollary}
\begin{definition}
A Lie algebra $\mathfrak{g}$ is called residually nilpotent if
$$\cap_{i=1}^{\infty}C^{s}\mathfrak{g}=0.$$
\end{definition}
\begin{definition}
The coclass of a Lie algebra $\mathfrak{g}$ (which can be equal to
infinity) is the natural number $cc(\mathfrak{g})$ that is defined
as $cc(\mathfrak{g})= \sum_{i\ge 1}(\dim
(C^i\mathfrak{g}/C^{i+1}\mathfrak{g})-1)$. Finitely generated
residually nilpotent Lie algebras of coclass $1$ is also called as
algebras of maximal class( infinite dimensional filiform Lie
algebras).
\end{definition}
\begin{example}
Let us define $L_k$ as the Lie algebra of poynomial vector fields
on the real line${\mathbb R}^1$, having the zero $x=0$ of order
not less than $k+1$.

The algebra $L_k$ can be defined  by its infinite basis and
structure relations
$$e_i=x^{i+1}\frac{d}{dx}, \; i \in {\mathbb N},\; i \ge k, \quad \quad
[e_i,e_j]= (j-i)e_{i{+}j}, \; \forall \;i,j \in {\mathbb N}.$$

Let consider $k=1$, it is simple to note, that $L_1$ is a
residually nilpotent Lie algebra of maximal class, generated by
two elements: $e_1$ and $e_2$.
\end{example}
We recall that ${\mathbb Z}$-graded Lie algebra $W$ defined by its
basis $e_i,\; i \in {\mathbb Z}$ and relations
$$[e_i.e_j]=(j-i)e_{i+j}, \quad \forall i,j \in {\mathbb Z},$$
is called the Witt algebra. The Lie algebra $L_1$ sometimes is
called the positive part $W_+=\oplus_{i>0}(W)_i$ of the Witt Lie
algebra.

Let us give the two additional examples infinite dimensional
${\mathbb N}$-graded Lie algebras of maximal class.
\begin{example}
Let us denote by $\mathfrak{m}_0$ the direct limit of Lie
algebras$\mathfrak{m}_0(n)$ when $n \to \infty$. This
infinitedimensional Lie algebra $\mathfrak{m}_0$ can be defined by
a basis $e_1, e_2, \dots, e_n, \dots $ and structure relations:
$$ \label{m_0}
[e_1,e_i]=e_{i+1}, \; \forall \; i \ge 2.$$
\end{example}
\begin{example}
A Lie algebra $\mathfrak{m}_2$ is defined by its infinite basis
$e_1, e_2, \dots, e_n, \dots $ and structure relations
$$
[e_1, e_i ]=e_{i+1}, \quad \forall \; i \ge 2; \quad \quad [e_2,
e_j ]=e_{j+2}, \quad \forall \; j \ge 3.
$$
\end{example}
We will call the filtration $C$ of a residually nilpotent Lie
algebra $\mathfrak g$ the canonical filtration of $\mathfrak g$.
\begin{remark}
There are the following isomorphisms ${\rm gr}_C{L_1}\cong {\rm
gr}_C\mathfrak{m}_2 \cong {\rm gr}_C\mathfrak{m}_0 \cong
\mathfrak{m}_0$.
\end{remark}
\begin{theorem}[Vergne \cite{V}]
\label{V_ne1} Let
$\mathfrak{g}=\oplus_{\alpha}\mathfrak{g}_{\alpha}$ be a ${\mathbb
N}$-graded $n$-dimensional filiform Lie algebra and
$$
\dim \mathfrak{g}_1=2, \quad \dim \mathfrak{g}_2=\ldots= \dim
\mathfrak{g}_{n{-}1}=1.
$$
Then

1) if $n=2k+1$, then $\mathfrak{g}$ is isomorphic to
$\mathfrak{m}_0(2k+1)$;

2) if $n=2k$, then $\mathfrak{g}$ is isomorphic either to
$\mathfrak{m}_0(2k)$ or to the algebra $\mathfrak{m}_1(2k)$
defined by the basis $e_1, \ldots, e_{2k}$ and structure realtions
$$
[e_1, e_i ]=e_{i+1}, \; i=2, \ldots, 2k{-}1, \quad \quad [e_j,
e_{2k{+}1{-}j} ]=(-1)^{j{+}k}e_{2k}, \quad j=2, \ldots, k.
$$
\end{theorem}
\begin{remark}
In the formulation of the theorem \ref{V_ne1} the Lie algebra
gradings $\mathfrak{m}_0(n)$, $\mathfrak{m}_1(n)$ are defined as
$\mathfrak{g}_1=\langle e_1, e_2 \rangle $,
$\mathfrak{g}_i=\langle e_{i{+}1} \rangle, \: i=2, \ldots, n{-}1.$
\end{remark}
Both the examples of infinite dimensional Lie algebras that we
have considered are nonaccidentally the direct limits of some
suites of filiform Lie algebras. It follows directly from the
definition, that an arbitrary Lie algebra of maximal class can be
considered as the direct limit of some infinite chain of inserted
in each other filiform Lie algebras, where each subsequent algebra
of the chain is obtained as a one-dimensional central extension of
the previous algebra. Specifically, such direct limits of
filiformLie algebras we will consider in the sequel. Constructing
by recursion a suite of the added elements $e_i$ with the central
extensions, one gets the following corollary:
\begin{corollary}
\label{pos_def} Let $\mathfrak g$ be finitely generated residually
nilpotent Lie algebra of maximal class. Then
$${\rm gr}_C{\mathfrak g}\cong \mathfrak{m}_0.$$
In $\mathfrak g$ one can choose an infinite basis $e_1, e_2,
\ldots, e_n, \dots $ such as
$$
[e_1, e_i]=e_{i+1}, \quad \forall i \ge 2,  \quad [e_i, e_j]=
   \sum_{k{=}0}^{N(i,j)}c_{ij}^{i{+}j{+}k} e_{i{+}j{+}k}.
$$
\end{corollary}
By means of some fixed basis $e_1, e_2, \ldots, e_n, \dots $ one
can defined another one (non canonical) decreasing filtration
$\tilde F$ of the algebra $\mathfrak g$:
$$
\tilde F^{k}\mathfrak{g}= \langle e_k,e_{k{+}1},e_{k{+}2},\dots
\rangle, \quad k \ge 1.
$$
In this case the homogeneous components of the corresponding
associated graded Lie algebra  ${\rm gr}_{\tilde F} {\mathfrak g}$
will be already one-dimensional.

The classification of such associated graded Lie algebras is
known.
\begin{theorem}[\cite{Fial, ShZ}]
\label{fialZ} Let
$\mathfrak{g}=\bigoplus_{i=1}^{\infty}\mathfrak{g}_{i}$ be a
${\mathbb N}$-graded Lie algebra of maximal class such that
$$
\dim \mathfrak{g}_i=1, \quad \forall i \ge 1.
$$
Then $\mathfrak{g}$ is isomorphic to one (and the only one) Lie
algebra of three given ones:
$$\mathfrak{m}_0, \; \mathfrak{m}_2, \; L_1.$$
\end{theorem}
\section{Lie algebra cohomology}
\label{cohomology} Let $\mathfrak{g}$ be a Lie algebra over a
field ${\mathbb K}$ and $\rho: \mathfrak{g} \to \mathfrak{gl}(V)$
be its linear representation (thevector space $V$ is a
$\mathfrak{g}$-module). Let us denote by $C^q(\mathfrak{g},V)$ the
space of $q$-linear skew-symmetric mappings  from $\mathfrak{g}$
to $V$. The one can consider the following algebraic complex:
$$
\begin{CD}
V @>{d_0}>> C^1(\mathfrak{g}, V) @>{d_1}>> C^2(\mathfrak{g}, V)
@>{d_2}>> \dots @>{d_{q{-}1}}>> C^q(\mathfrak{g}, V) @>{d_q}>>
\dots
\end{CD},
$$
where the differential $d_q$ is defined as:
\begin{equation}
\begin{split}
(d_q f)(X_1, \dots, X_{q{+}1})= \sum_{i{=}1}^{q{+}1}(-1)^{i{+}1}
\rho(X_i)(f(X_1, \dots, \hat X_i, \dots, X_{q{+}1}))+\\
+ \sum_{1{\le}i{<}j{\le}q{+}1}(-1)^{i{+}j{-}1} f([X_i,X_j],X_1,
\dots, \hat X_i, \dots, \hat X_j, \dots, X_{q{+}1}).
\end{split}
\end{equation}
The cohomology of the complex $(C^*(\mathfrak{g}, V), d)$ is
called the cohomology of the Lie algebra $\mathfrak{g}$ with
coefficients in the representation $\rho: \mathfrak{g} \to V$.

Further we will examine two following representations:

1) $V= {\mathbb K}$ and the morphism  $\rho: \mathfrak{g} \to
{\mathbb K}$ is trivial.

2) $V= \mathfrak{g}$ and the adjoint representation $\rho=ad:
\mathfrak{g} \to \mathfrak{g}$.

The cohomology of the complex $(C^*(\mathfrak{g}, {\mathbb K}),
d)$ are called the cohomology with trivial coefficients of the Lie
algebra $\mathfrak{g}$ and it is denoted by $H^*(\mathfrak{g})$.
We will fix the notation $H^*(\mathfrak{g},\mathfrak{g})$ for the
cohomology of a Lie Algebra $\mathfrak{g}$ with coefficients in
the adjoint representation.
\begin{example}
Let $\mathfrak{g}$ be a  ${\mathbb N}$-graded Lie algebra defined
by its infuinite basis $e_1, e_2, \dots, e_n, \dots$ and structure
relations
$$[e_i,e_j]= c_{ij}e_{i{+}j}.$$
Let us examine the dual basis $e^1, e^2, \dots, e^n, \dots$. One
can consider a grading (that we will call the weight) of
$\Lambda^*(\mathfrak{g}^*)=C^*(\mathfrak{g},{\mathbb K})$:
$$\Lambda^* (\mathfrak{g}^*)=
\bigoplus_{\lambda{=}1}^{\infty} \Lambda^*_{(\lambda)}
(\mathfrak{g}^*),$$ where the subspace $\Lambda^{q}_{(\lambda)}
(\mathfrak{g}^*)$ is spanned by  $q$-forms $\{ e^{i_1} {\wedge}
\dots {\wedge} e^{i_q}, \; i_1{+}\dots{+}i_q {=} \lambda \}$. For
instance a monomial $e^{i_1} \wedge \dots \wedge e^{i_q}$ has the
degree $q$ and the weight $\lambda=i_1{+}\dots{+}i_q$.

In its turn the complex $(C^*(\mathfrak{g}, \mathfrak{g}), d)$ is
${\mathbb Z}$-graded:
$$C^*(\mathfrak{g}, \mathfrak{g})=
\bigoplus_{\mu \in {\mathbb Z}}C^*_{(\mu)}(\mathfrak{g},
\mathfrak{g}),$$ where $C^q_{(\mu)}(\mathfrak{g}, \mathfrak{g})$
spanned by monomials $\{ e_l \otimes e^{i_1} {\wedge} \dots
{\wedge} e^{i_q}, \; i_1{+}\dots{+}i_q{+}\mu =l \}$.
\end{example}
The cohomology algebra $H^*(\mathfrak{m}_0)$ was computed in
\cite{FialMill}.

In order to formulate the main result  \cite{FialMill} we will
need the following linear operators who act on the exterior
algebra $\Lambda^*(e_2, e_3, \dots)$ with generators $e_2, e_3,
\dots$:

1) $D_1: \Lambda^*(e_2, e_3, \dots) \to \Lambda^*(e_2, e_3,
\dots)$,
\begin{equation}
\begin{split}
D_1(e^2)=0, \; D_1(e^i)= e^{i-1}, \; \forall i\ge 3,\\
D_1(\xi\wedge \eta)=D_1(\xi)\wedge \eta +\xi\wedge D_1(\eta), \;
\: \forall \xi, \eta \in \Lambda^*(e^2, e^3, \dots).
\end{split}
\end{equation}

2) and its right-inverse operator $D_{-1}:
\Lambda^*(e^2,e^3,\dots) \to \Lambda^*(e^2,e^3,\dots)$,
\begin{equation}
\begin{split}
\label{D_{-1}} e^i=e^{i+1}, \; D_{-1}(\xi{\wedge} e^i)= \sum_{l\ge
0}(-1)^l D_{1}^{l}(\xi){\wedge} e^{i+1+l},
\end{split}
\end{equation}
where $i\ge 2$ and $\xi$ is an arbitrary form in
$\Lambda^*(e^2,\dots,e^{i-1})$. The sum in the definition
(\ref{D_{-1}}) of the operator $D_{-1}$ is always finite, as.
$D_1^l$ decreases the second grading by  $l$. For instance

$$D_{-1}(e^i\wedge e^k)=\sum_{l=0}^{i-2} ({-}1)^l e^{i-l}{\wedge} e^{k+l+1}.$$
\begin{proposition}
The operators $D_1$ and $D_{-1}$ has the following properties:
$$
\label{D_1,D_{-1}} d\xi=e^1\wedge D_1\xi, \; e^1\wedge \xi=d
D_{-1}\xi, \; D_1D_{-1}\xi=\xi, \quad \xi \in
\Lambda^*(e^2,e^3,\dots).
$$
\end{proposition}
\begin{remark}
The operator $D_1$ is in fact the operator $ad^*e_1$ that is the
adjoint operator to $ad e_1: (e_2,e_3,\dots) \to (e_2,e_3,\dots)$
extended from $(e_2,e_3,\dots)^*$ to the whole exterior algebra
$\Lambda^*(e^2, e^3, \dots)$ like a derivation of degree zero.
\end{remark}
\begin{theorem}[\cite{FialMill}]
The infinite dimensional bigraded cohomology
$H^*(\mathfrak{m}_0)=\oplus_{k,q} H^q_k(\mathfrak{m}_0)$ is
spanned by the cohomology classes of the linear forms $e^1$, $e^2$
and the following homogeneous cocycles:
\begin{equation}
\label{cocycles} \omega(e^{i_1}{\wedge}\dots {\wedge} e^{i_q}
{\wedge} e^{i_q{+}1})= \sum\limits_{l\ge 0}(-1)^l
D_1^l(e^{i_1}\wedge \dots \wedge e^{i_q})\wedge e^{i_q+1+l},
\end{equation}
where $q \ge 1, \; 2\le i_1 {<}i_2{<}{\dots} {<}i_q$.
\end{theorem}
The formula (\ref{cocycles}) defines a homogeneous closed
$(q{+}1)$-form with the second grading equal to
$i_1{+}{\dots}{+}i_{q{-}1}{+}2i_{q}{+}1$. There is the only one
monomial in its  decomposition of the form $\xi \wedge e^i \wedge
e^{i+1}$ and it is equal to $e^{i_1}{\wedge}\dots {\wedge} e^{i_q}
{\wedge} e^{i_q{+}1}$.

The total number of linear independent $q$-cocycles with the
second grading $k{+}\frac{q(q{+}1)}{2}$ is equal to
\begin{equation}
\label{chislo_cocyclov}
{\rm dim}
H_{k{+}\frac{q(q{+}1)}{2}}^q(\mathfrak{m}_0)= P_q(k)-P_q(k{-}1),
\end{equation}
where $P_q(k)$ denots the number (not ordered) partitions of a
natural number $k$ into $q$ parts.
\begin{example}
One can choose the following basis in $H^2(\mathfrak{m}_0)$:
$$
e^2{\wedge} e^3, e^3{\wedge} e^4 -e^2{\wedge} e^5, \dots,
\omega(e^j{\wedge}e^{j+1})= \sum_{l=0}^{j-2} ({-}1)^l e^{j-l}
{\wedge} e^{j+1+l}, \dots
$$
\end{example}
\begin{example}
In $H^3(\mathfrak{m}_0)$ the following cocycles  form a basis:
\label{3cocycle}
$$
\omega(e^i{\wedge}e^j {\wedge}e^{j{+}1})= \sum\limits_{l\ge
0}(-1)^l D_1^l(e^{i}\wedge e^{j})\wedge e^{j{+}1{+}l}, \quad 2\le
i{<}j.
$$
In paticular we have for the cocyle
$\omega(e^5{\wedge}e^6{\wedge}e^7)$:
\begin{multline*}
\omega(e^5{\wedge} e^6 {\wedge} e^7)= e^5{\wedge} e^6
{\wedge}e^7-e^4{\wedge} e^6 {\wedge} e^8+(e^3{\wedge} e^6+
e^4 {\wedge} e^5){\wedge} e^9-\\
-(e^2{\wedge}e^6+2e^3 {\wedge}e^5){\wedge}e^{10} +(3e^2{\wedge}
e^5+2e^3{\wedge} e^4){\wedge} e^{11}- 5e^2{\wedge} e^4{\wedge}
e^{12}+5e^2{\wedge} e^3{\wedge} e^{13}.
\end{multline*}
\end{example}
The cohomology $H^*({\mathfrak{m}_0}, {\mathfrak{m}_0})$ were
calculated in \cite{Mill3}. We will not completely give the
appropriate result. For the purposes of the present work we will
need only "non-negative"\: cohomology $H^2_+({\mathfrak{m}_0},
{\mathfrak{m}_0})=\oplus_{s \ge 0}
H^2_s(\mathfrak{m}_0,\mathfrak{m}_0)$ and $H^3_+({\mathfrak{m}_0},
{\mathfrak{m}_0})=\oplus_{s \ge 0}
H^3_s(\mathfrak{m}_0,\mathfrak{m}_0)$.
\begin{theorem}[\cite{Mill3}]
\label{main} The graded spaces
$H^2_+(\mathfrak{m}_0,\mathfrak{m}_0)$ and
$H^3_+(\mathfrak{m}_0,\mathfrak{m}_0)$ are infinte dimensioal
spaces of formal series $\sum x_{j,s}\Psi_{j,s}$ and $\sum
x_{i,j,s}\Psi_{i,j,s}$ respectively and  homogeneous cocycles
$\Psi_{j,s} \in H^2_s(\mathfrak{m}_0,\mathfrak{m}_0)$ and
$\Psi_{i,j,s} \in H^3_s(\mathfrak{m}_0,\mathfrak{m}_0)$ are
defined by the formulas:
\begin{equation} \label{formulirovka}
\begin{split}
\Psi_{j,s}=\sum_{k=0}^{\infty}e_{2j{+}1{+}s{+}k} \otimes
D_{-1}^k\omega(e^j {\wedge}e^{j{+}1}), \; 2\le j, \;
s \ge 0;\\
\Psi_{i,j,s}=\sum_{k=0}^{\infty}e_{i{+}2j{+}1{+}s{+}k}
\otimes D_{-1}^k\omega(e^i{\wedge}e^j {\wedge}e^{j{+}1}), \; 2\le
i <j, \; s \ge 0.
\end{split}
\end{equation}
$2$-form $\Psi_{j,s}$ is uniquely determined  by the fact that it
is closed and by the following condition:
\begin{equation}
\label{property_Psi_1}
\begin{split}
\Psi_{j,s} \left(e_j,e_{j+1}\right)=e_{2j{+}1{+}s},
\\
\Psi_{j,s} \left(e_k,e_{k+1}\right)=0, \quad 2\le k \ne j.
\end{split}
\end{equation}
The cocycle $\Psi_{i,j,s}$ is also uniquely determined by the
similar conditions:
\begin{equation}
\label{property_Psi_2}
\begin{split}
\Psi_{i,j,s} \left(e_i,e_j,e_{j+1}\right)=e_{i{+}2j{+}1{+}s},
\\
\Psi_{i,j,s} \left(e_l,e_k,e_{k+1}\right)=0, \quad 2\le l{<}k, \;
l \ne i, k \ne j.
\end{split}
\end{equation}
\end{theorem}
It follows from the properties of the operator$D_{-1}$ that
$$
D_{-1}^k \omega(e^j \wedge e^{j+1})= \sum_{l=0}^{j-2}(-1)^l\binom
{k+l}{l}e^{j-l}\wedge e^{j+1+l+k},
$$
which makes it possible to rewrite the formula of $\Psi_{j,s}$ in
the following way:
\begin{equation}
\Psi_{j,s}=\sum_{l=0}^{+\infty}\sum_{r=0}^{j-2}(-1)^r\binom{l+r}{r}e_{s{+}l}
\otimes e^{j-r}\wedge e^{j+1+r+l}.
\end{equation}
It is equivalent to define the cocycle $\Psi_{j,s}$ by the table
of its values:
\begin{equation}
\label{znachenia_cocycla}
\Psi_{j,s}\left(e_{k},e_{m}\right)=(-1)^{j-k}\binom{m-j-1}{j-k}e_{m+k-2j-1+s},
\; 2 \le k \le j < m,\; k+m \ge 2j+1,
\end{equation}
where on the remaining vector pairs$e_{k},e_{m}$ the cocycle
$\Psi_{j,s}$ vanishes.
\begin{remark}
\label{khakim} The last formula shows, that in the finite
dimensional case our cocyle $\Psi_{j,s}$ coincides with the
cocycle $\Psi_{j,2j+1+s}$ from the basis of the subspace
$H^2_+({\mathfrak{m}_0}(n), {\mathfrak{m}_0}(n))$ \cite{Kh}.
\end{remark}

\section{The Variety of non-negative deformations of the algebra ${\mathfrak m}_0$ and nilpotent versal
deformation}
\label{Nijenhuis-Richardson}
\begin{definition}[\cite{NR}]
Let $\mathfrak g$ be a Lie algebra with the bracket
$[\cdot,\cdot]$ and $\Psi: {\mathfrak g} \otimes {\mathfrak g} \to
{\mathfrak g}$ be a skew-symmetric bilinear map. $\Psi$ is called
a deformation of the bracket $[\cdot,\cdot]$ if
$[\cdot,\cdot]'=[\cdot,\cdot]+\Psi$ defines a structure of a Lie
algebra on the vector space ${\mathfrak g}$.
\end{definition}

The Jacobi identity for the bracket
$[\cdot,\cdot]'=[\cdot,\cdot]+\Psi$
$$\left[[x,y]',z\right]'+\left[[y,z]',x\right]'+\left[[z,x]',y\right]'=0$$
is equivalent to the so-called deformation equation
\begin{equation}
\label{mdef} d\Psi+\frac{1}{2}[\Psi,\Psi]=0,
\end{equation}
where the bracket $[\cdot,\cdot]$ denotes this time the
Nijenhuis-Richardson bracket $C^2(\mathfrak g,\mathfrak g)\times
C^2(\mathfrak g,\mathfrak g) \to C^3(\mathfrak g,\mathfrak g)$:
\begin{equation}
\label{skobka_NR}
\begin{split}
[\Psi,\tilde \Psi](x,y,z)=\Psi(\tilde \Psi(x,y),z)+\Psi(\tilde \Psi(y,z),x)+
\Psi(\tilde \Psi(z,x),y)+ \\
+\tilde \Psi(\Psi(x,y),z)+\tilde \Psi(\Psi(y,z),x)+\tilde \Psi(\Psi(z,x),y).
\end{split}
\end{equation}

The last operation is extended to the bracket on the total space
$C^*(\mathfrak g,\mathfrak g)$:
$$
[\cdot,\cdot]: C^p(\mathfrak g,\mathfrak g)\times C^q(\mathfrak
g,\mathfrak g) \to C^{p{+}q{-}1}(\mathfrak g,\mathfrak g).$$
Specifically for $\alpha \in C^p(\mathfrak g,\mathfrak g)$
and$\beta \in C^q(\mathfrak g,\mathfrak g)$ one can define
$[\alpha,\beta] \in C^{p{+}q{-}1}(\mathfrak g,\mathfrak g)$:
\begin{equation}
\begin{split}
[\alpha{,}\beta](\xi_1,{\small \ldots},\xi_{p{+}q{-}1}){=}
\sum\limits_{1{\le}i_1{<}{\small \ldots}{<}i_q{\le}p{+}q{-}1}
\alpha(\beta(\xi_{i_1}{,}{\small
\ldots}{,}\xi_{i_q})\xi_1{,}{\small \ldots}{,} \hat
\xi_{i_1}{,}{\small \ldots}{,}
\hat \xi_{i_q}{,} {\small \ldots}{,}\xi_{p{+}q{-}1}){+}\\
+({-}1)^{pq{+}p{+}q}\sum\limits_{1{\le}j_1{<}{\ldots}{<}j_p{\le}p{+}q{-}1}
\beta(\alpha(\xi_{j_1}{,}{\small
\ldots}{,}\xi_{j_p})\xi_1{,}{\small \ldots}{,} \hat
\xi_{j_1}{,}{\small \ldots}{,} \hat \xi_{j_q}{,} {\small
\ldots}{,}\xi_{p{+}q{-}1}).
\end{split}
\end{equation}

The Nijenhuis-Richardson bracket determines a Lie superalgebra
structure on $C^*(\mathfrak g,\mathfrak g)$, i.e. if $\alpha \in
C^p(\mathfrak g,\mathfrak g)$, $\beta \in C^q(\mathfrak
g,\mathfrak g)$ and $\gamma \in C^r(\mathfrak g,\mathfrak g)$ then
\begin{equation}
\label{superalg}
\begin{split}
\quad [\alpha,\beta]={-}({-}1)^{(p{-}1)(q{-}1)}[\beta,\alpha];
\hspace{30mm}\\  ({-}1)^{(p{-}1)(q{-}1)}
\left[[\alpha,\beta],\gamma\right]{+}
({-}1)^{(q{-}1)(r{-}1)}\left[[\beta,\gamma],\alpha\right]{+}
({-}1)^{(r{-}1)(p{-}1)}\left[[\gamma,\alpha],\beta\right]= 0.
\end{split}
\end{equation}

in addition to this the Nijenhuis-Richardson bracket is compatible
with $d$ of the cochain complex $C^*(\mathfrak g,\mathfrak g)$:
$$d[\alpha,\beta]=[d\alpha,\beta]+({-}1)^{p}[\alpha,d\beta].$$
Hence the Nijenhuis-Richardson bracket determines a Lie
superalgebra structure in the cohomology $H^*(\mathfrak
g,\mathfrak g)$ also:
$$
[\cdot,\cdot]: H^p(\mathfrak g,\mathfrak g)\times H^q(\mathfrak
g,\mathfrak g) \to H^{p{+}q{-}1}(\mathfrak g,\mathfrak g)$$ with
properties(\ref{superalg}).

\begin{proposition}
Let $\mathfrak{g}= \oplus_{\alpha} \mathfrak{g}_{\alpha}$ be
${\mathbb N}$-graded Lie algebra. In this case the ${\mathbb
Z}$-gradings of $C^*(\mathfrak{g},\mathfrak{g})$ and
$H^*(\mathfrak{g},\mathfrak{g})$ are compatible with the
Nijenhuis-Richardson bracket:
\begin{equation}
\begin{split}
[\cdot,\cdot]: C^p_{(\mu)}(\mathfrak{g},\mathfrak{g}) \times
C^q_{(\nu)}(\mathfrak{g},\mathfrak{g}) \longrightarrow
C^{p{+}q{-}1}_{(\mu+\nu)}(\mathfrak{g},\mathfrak{g})\\
[\cdot,\cdot]: H^p_{(\mu)}(\mathfrak{g},\mathfrak{g}) \times
H^q_{(\nu)}(\mathfrak{g},\mathfrak{g}) \longrightarrow
H^{p{+}q{-}1}_{(\mu+\nu)}(\mathfrak{g},\mathfrak{g})
\end{split}
\end{equation}
\end{proposition}

Further we will examine only deformations$\Psi$ of ${\mathbb
N}$-graded Lie algebra $\mathfrak{g}=\oplus_{\alpha
> 0} \mathfrak{g}_{\alpha}$ of the form
$$
\Psi=\Psi_{0}+\Psi_{1}+\Psi_{2}+\ldots+\Psi_i+\ldots, \qquad
\Psi_i \in C^2_{(i)}(\mathfrak{g},\mathfrak{g}), \quad
i=1,2,\ldots.
$$

Decomposing $\Psi=\Psi_{0}+\Psi_{1}+\Psi_2+\ldots$ in the
deformation equation (\ref{mdef}) to uniform terms and comparing
the terms with same grading we come to the following system of
equations on homogeneous components $\Psi_l$:
\begin{equation} \label{sysdef}
\begin{split}
 d\Psi_{0}+\frac{1}{2}[\Psi_{0},\Psi_{0}]=0, \quad
d\Psi_{1}+[\Psi_{0},\Psi_{1}]=0, \quad
d\Psi_{2}+[\Psi_{0},\Psi_{2}]+\frac{1}{2}[\Psi_{1},\Psi_{1}]=0,\ldots, \\
d\Psi_{i}+\frac{1}{2}\sum_{m{+}l{=}i}[\Psi_{m},\Psi_{l}]=0, \ldots
\hspace{19mm}
\end{split}
\end{equation}
\begin{corollary}
The set $\Psi_{0},\Psi_{1}=0,\Psi_2=0,\ldots$ is a solution of the
system (\ref{sysdef}) and hence  $[,]+\Psi_0$ determines a new Lie
bracket.
\end{corollary}
\begin{proposition}
\label{neotr_def} An arbitrary infinite dimensional Lie algebra of
maximal class ${\mathfrak g}$ is isomorphic to some non-negative
deformation $\Psi$ of the graded Lie algebra  ${\mathfrak m}_0$
$$
\Psi=\Psi_0+\Psi_1+\Psi_2+\dots+\Psi_i+\dots,\quad \Psi_i \in
C^2_{(i)}(\mathfrak{m}_0,\mathfrak{m}_0)
$$
such that
\begin{equation}
\label{adapt_basis}
\Psi(e_1, e_k)=0, \forall k \in {\mathbb N}.
\end{equation}
\end{proposition}
\begin{proof}
The statement is just a reformation of the Theorem \ref{fialZ} in
the terms of deformation theory.
\end{proof}
We will call deformations $\Psi$ of the form (\ref{adapt_basis})
adapted deformations.
\begin{proposition}
A form $\Psi \in C^2_+({\mathfrak m}_0,{\mathfrak m}_0)$  which
satisfies (\ref{kogranica}) is closed: $d\Psi=0$. The deformation
equation on the cocycle $\Psi$ will be written down as
$\left[\Psi, \Psi \right]=0$:
\begin{equation} \label{sysdef_2}
\begin{split}
\frac{1}{2}[\Psi_{0},\Psi_{0}]=0, \quad [\Psi_{0},\Psi_{1}]=0,
\quad
[\Psi_{0},\Psi_{2}]+\frac{1}{2}[\Psi_{1},\Psi_{1}]=0,\ldots, \\
\frac{1}{2}\sum_{m{+}l{=}i}[\Psi_{m},\Psi_{l}]=0, \ldots
\hspace{19mm}
\end{split}
\end{equation}
\end{proposition}
\begin{proof}
In order to verify the closeness of the form $\Psi$ it it suffices
to examine the values$d\Psi$ on the triples of basic vectors. From
the other side if among triples $e_i,e_j,e_k$ there are no $e_1$
then $d\Psi(e_i,e_j,e_k)=0$ in view of the commutation
relationships ${\mathfrak m}_0$ but if, for example, $e_i=e_1$,
then in view of $\Psi(e_1, e_k)=0, \forall k$ it holds
$$d\Psi(e_1,e_j,e_k)=-\left[e_1,\Psi(e_j,e_k) \right].$$
The kernel of the operator $ad e_1: {\mathfrak m}_0 \to {\mathfrak
m}_0$ is one-dimensional and spanned by $e_1$ and hence  $d\Psi=0$
if and only if $\Psi=0$ ($\Psi(e_j,e_k)\ne \alpha e_1$ because
$\Psi \in C^2_+({\mathfrak m}_0,{\mathfrak m}_0)$).
\end{proof}
as it was already noted earlier the cocycle $\Psi$ can be
decomposed in a formal series $\Psi=\sum x_{j,s}\Psi_{j,s}$ with
respect to basic cocycles $\Psi_{j,s}$. But it is easy to see that
not eveyone formal series answers a Lie algebra in a standard
meaning.
\begin{proposition}
A formal series $\Psi=\sum x_{j,s}\Psi_{j,s}$ determines a Lie
algebra structure $[,]+\Psi$ if and only if  $[\Psi,\Psi]=0$ and
for an arbitrary $j \in {\mathbb N}$ exists $N(j)\in {\mathbb N}$
such that $x_{j,s}=0$ when $s > N(j)$.
\end{proposition}
\begin{proof}
Let us consider a commutator
$[e_j,e_{j+1}]_{\Psi}=\Psi(e_j,e_{j+1})=\sum_{s=0}^{+\infty}
x_{j,s}e_{2j+1+s}$ when $j>1$. Our condition is equivalent to the
fact that the linear combination is finite and it is necessary
condition. From the other side, if it satisfied, for an arbitrary
commutator $[e_k,e_m]_{\Psi}, m>k
>1,$ one has according to the formula (\ref{znachenia_cocycla})
$$
[e_k,e_m]_{\Psi}=\sum_{j,s}
x_{j,s}\Psi_{j,s}(e_k,e_m)=\sum_{j=k}^{\left[\frac{m+k+1}{2}\right]}
(-1)^{j-k}\binom{m-j-1}{j-k}\sum_{s=0}^{N(j)}x_{j,s}e_{m+k-2j-1+s}.
$$
\end{proof}
\begin{theorem}
\label{main_theorem} Let us consider a fromal series in the
completed space $\oplus_{s \ge 0}
H^2_s(\mathfrak{m}_0,\mathfrak{m}_0)$
$$\Psi=\sum_{s=0}^{+\infty}\sum_{j=2}^{+\infty}x_{j,s}\Psi_{j,s},$$
where cocycles $\{\Psi_{j,s}\}$ are defined by the formulas
(\ref{znachenia_cocycla}). A cocycle $\Psi$ satisfies the
deformation equation $[\Psi,\Psi]=0$ if and only if the
coefficients $x_{j,s}$ of the series $\Psi$ satisfy the following
system of quadratic equations $F_{j,q,r}=0,\;2 \le j {<} q, r
{\ge} 0$:
\begin{equation}
\label{main_equations}
\begin{split}
F_{j,q,r}=\sum_{t{=}0}^r
\sum_{l{=}j}^{\left[\frac{j{+}q{-}1}{2}\right]}
\sum_{m{=}q{+}1}^{q{+}\left[\frac{j{+}t}{2}\right]}({-}1)^{l{-}j{+}m{-}q}
\binom{q{-}l{-}1}{l{-}j}\binom{j{+}q{-}m{+}t{-}1}{m{-}q{-}1}x_{l,t}x_{m,r{-}t}+\\
+\sum_{t{=}0}^r \sum_{l{=}j}^{\left[\frac{j{+}q}{2}\right]}
\sum_{m{=}q}^{q{+}\left[\frac{j{+}t}{2}\right]}({-}1)^{l{-}j{+}m{-}q}
\binom{q{-}l}{l{-}j}\binom{j{+}q{-}m{+}t}{m{-}q}x_{l,t}x_{m,r{-}t}+\\
+\sum_{t{=}0}^r
\sum_{m{=}j}^{q{+}\left[\frac{j{+}t}{2}\right]}({-}1)^{m{-}j{+}1}
\binom{2q{-}m{+}t}{m{-}j}x_{q,t}x_{m,r{-}t}=0.
\end{split}
\end{equation}
\end{theorem}
\begin{proof}
earlier it was noted, that the Nijenhuis-Richardson bracket is
compatible with the differentuial of the Lie algebra cochain
complex with coefficients in the adjoint representation. Thus the
square $[\Psi, \Psi]$ of an arbitrary cocyle $\Psi$ is a closed
form. In addition to this it follows from $\Psi(e_1,e_k)=0,
\forall k$ that $[\Psi,\Psi](e_1,e_k,e_l)=0, \forall k,l$. But a
cocycle that satisfies this condition uniquely determines some
cohomology in $\oplus_{s \ge 0}
H^3_s(\mathfrak{m}_0,\mathfrak{m}_0)$ ( see the proof of the
Proposition (\ref{neotr_def})) and hence $\frac{1}{2}[\Psi, \Psi]$
can be decomposed in a formal serieswith respect to the basis
$\Psi_{j,q,r}$:
$$
\frac{1}{2}[\Psi,\Psi]=\sum_{2 \le j < q, r \ge
0}F_{j,q,r}\Psi_{j,q,r}.
$$
Thus $[\Psi,\Psi]=0$ if when and only when all the coefficients
$F_{j,q,r}$ of this decomposition vanish. It remains to calculate
them explicitly as polynomials on variables $x_{j,s}$. Let us
remark for this that according to the property
(\ref{property_Psi_2})
$$
\frac{1}{2}[\Psi,\Psi](e_j,e_q,e_{q+1})=\sum_{r=0}^{+\infty}F_{j,q,r}e_{j+2q+1+r}.
$$
On the another side
$$
\frac{1}{2}[\Psi,\Psi]=\sum_{r=0}^{+\infty}\sum_{s+t=r}\sum_{m\ge
l} x_{m,s}x_{l,t}[\Psi_{m,s},\Psi_{l,t}].
$$
Let us calculate the value
$[\Psi_{m,s},\Psi_{l,t}](e_j,e_q,e_{q+1})$ when $j < q$ directly
from the definition (\ref{skobka_NR}) of the Nijenhuis-Richardson
bracket under the assumption that $m > l$:
\begin{equation}
\label{skobka_basisa}
\begin{split}
[\Psi_{m,s}{,}\Psi_{l,t}](e_j,e_q,e_{q{+}1})=\Psi_{m,s}(
\Psi_{l,t}(e_j,e_q),e_{q{+}1}){+}\Psi_{m,s}(
\Psi_{l,t}(e_q,e_{q{+}1}),e_j){+}
\Psi_{m,s}(\Psi_{l,t}(e_{q{+}1},e_j),e_q){+} \\
+\Psi_{l,t}(\Psi_{m,s}(e_j,e_q),e_{q{+}1}){+}
\Psi_{l,t}(\Psi_{m,s}(e_q,e_{q{+}1}),e_j){+}
\Psi_{l,t}(\Psi_{m,s}(e_{q{+}1},e_j),e_q).
\end{split}
\end{equation}
Using the properties of the cocycles $\Psi_{i,r}$ we have when
$t\ge 0, s \ge 0$:
$$
\Psi_{m,s}(
\Psi_{l,t}(e_j,e_q),e_{q{+}1})=(-1)^{l{-}j{+}m{-}q}\binom{q{-}l{-}1}{l{-}j}\binom{j{+}q{-}m{+}t{-}1}{m{-}q{-}1}
e_{2q{+}j{+}1{+}t{+}s},
$$
assuming that $0\le l{-}j\le q{-}l{-}1$ and $0\le m{-}q{-}1\le
j{+}q{-}m{+}t{-}1$, otherwise the value iz trivial. We fin
analogously the third summand:
$$
\Psi_{m,s}(
\Psi_{l,t}(e_{q{+}1},e_j),e_q)=(-1)^{l{-}j{+}m{-}q}\binom{q{-}l}{l{-}j}\binom{j{+}q{-}m{+}t}{m{-}q}
e_{2q{+}j{+}1{+}t{+}s},
$$
under $0\le l{-}j\le q{-}l$ and $0\le m{-}q\le j{+}q{-}m{+}t$. It
is it is simple to be remark that the forth and the sixth summands
in (\ref{skobka_basisa}) vanish when  $m> l$. Concerning the
second and the fifth summands, they are non-trivial only when
$l=q$ and $m=q$ respectively. In this case
\begin{equation}
\begin{split}
\Psi_{m,s}(
\Psi_{q,t}(e_q,e_{q{+}1}),e_j)={-}(-1)^{m{-}j}\binom{2q{-}m{+}t}{m{-}j}
e_{2q{+}j{+}1{+}t{+}s},\\
 \Psi_{l,t}(\Psi_{q,s}(e_q,e_{q{+}1}),e_j)={-}(-1)^{l{-}j}\binom{2q{-}l{+}s}{l{-}j}
e_{2q{+}j{+}1{+}t{+}s},
\end{split}
\end{equation}
where if $l{=}q$ then it holds on that $0\le m{-}j\le 2q{-}m{+}t$
and $m>l{=}q$ with $m{=}q$ the inequalities $0\le l{-}j\le
2q{-}l{+}s$ and $l{<} m{=}q$ must hold on. The contribution of
these terms with $m>l$ to the total sum (\ref{main_equations}) is
expressed in the following way
$$
 -\sum_{t{=}0}^r \sum_{l{=}j}^{q{-}1}({-}1)^{l{-}j}
\binom{2q{-}l{+}r{-}t}{l{-}j}x_{l,t}x_{q,r{-}t}-\sum_{t{=}0}^r
\sum_{m{=}q{+}1}^{q{+}\left[\frac{j{+}t}{2}\right]}({-}1)^{m{-}j}
\binom{2q{-}m{+}t}{m{-}j}x_{q,t}x_{m,r{-}t}=0.
$$
after redesignating in the second sum the indices of the summing
up from $m, t$ to $l,r-t$ respectively and after adding the terms
with $m=l=q$ we will get the sum from the third line of the
formula (\ref{main_equations}).

The case  $m=l$ is investigated analogously. gathering together
the results of the calculations we get the expression for the
coefficient $F_{j,q,r}$ with $e_{2q{+}j{+}1{+}r}$ standing in the
decomposition $[\Psi,\Psi](e_j,e_q,e_{q{+}1})$.
\end{proof}
A natural analogy with the singularity theory says that the affine
variety  $M_{\it Fil}$ is nothing else that a variety of
parameters  of nilpotent versal deformation of the algebra
${\mathfrak m}_0$: an arbitrary residually nilpotent Lie algebra
${\mathfrak g}$ of maximal class  is isomoorphic to some algebra
from our variety $M_{\it Fil}$. The last property is called
versality (in contrast to the universality) because a pointof the
variety $M_{\it Fil}$ is determined for an algebra  ${\mathfrak
g}$ not uniquely.
\begin{proposition}
\label{isom_lemma} Two deformations
$\Psi=\Psi_0+\Psi_1+\Psi_2+\dots+\Psi_i+\dots$ and $\tilde
\Psi=\beta\Psi_0+\beta\alpha\Psi_1+\beta\alpha^2\Psi_2+\dots+\beta\alpha^i\Psi_i+\dots$
of the Lie algebra ${\mathfrak m}_0$ determine isomorphic Lie
algebras with $\alpha \ne 0, \beta \ne 0$.
\end{proposition}
\begin{proof}
it suffices to examine an automorphism $\varphi$ of the algebra
${\mathfrak m}_0$ of the following form:
$$
\varphi(e_i)=\beta \alpha^i e_i, \: i \in {\mathbb N}.
$$
\end{proof}
In connection with nilpotent algebras our definition of nilpotent
versal deformation appears more convenient than more general
abstract definition \cite{FialFu}. But if we nevertheless follow
the abstract approach, let us consider the quotient of an
associative commutative polynomial algebra
$$A={\mathbb K}\left[[\{x_{j,s}\}]\right]/(\{F_{j,q,r}\})$$
from the infinite collection of the variables $\{x_{j,s}, j \ge 2,
s \ge 0\}$ with the unit $1$ over ideal generated by the infinite
system of quadratic polynomials$\{F_{j,q,s}, 2\le j<q, s \ge 0\}$
of this algebra? defined by formulas (\ref{main_equations}). The
augmentation $\varepsilon: A \to {\mathbb K}$  of the algebra $A$
is defined in a standard way: $\varepsilon(P)=P(0)$, where a
polynomial
 $P \in A$. The completed tensor product
$A\hat{\otimes} {\mathfrak m}_0$ is not only a linear space over
${\mathbb K}$ but and a  $A$-module: for all $x_{j,s} \in A$ and
$x \in {\mathfrak m}_0$ we suppose that $x_{j,s}\cdot 1\otimes x
=x_{j,s}\otimes x$.

The theorem \ref{main_theorem} can be formulated in the new terms.
\begin{proposition}
An $A$-linear Lie bracket $[,]_A$ is defined on the completed
tensor product $A\hat{\otimes} {\mathfrak m}_0$  which is assigned
as follows
$$
\left[ 1 \otimes x, 1 \otimes y \right]_A=1 \otimes
[x,y]_{{\mathfrak m}_0}+\sum_{j,s}x_{j,s}\otimes \Psi_{j,s}(x,y),
$$
where the values $\Psi_{j,s}(x,y) \in {\mathfrak m}_0$ are defined
by the formula (\ref{znachenia_cocycla}). It is evident that an
application
$$
\varepsilon \otimes id : A \otimes {\mathfrak m}_0 \to {\mathbb K}
\otimes {\mathfrak m}_0={\mathfrak m}_0
$$
is a Lie ${\mathbb K}$-algebra homomorphism.
\end{proposition}
The completed tensor product $A\hat{\otimes} {\mathfrak m}_0$ with
bracket$[,]_A$ is called a deformation of the algebra ${\mathfrak
m}_0$ with the base $A$. It is possible to call the corresponding
affine variety $M_{\it Fil}$ \cite{FialFu} the base of such
deformation. It is evident that $A\hat{\otimes} {\mathfrak m}_0$
is a residually nilpotent Lie $A$-algebra of maximal class. We
will not formulate here the versality condition in terms of formal
algebra referring for details to \cite{FialFu}. We will call also
the Lie $A$-algebra $A\otimes {\mathfrak m}_0$ as nilpotent versal
deformation of the algebra ${\mathfrak m}_0$.
\begin{remark}
let us extract three first equations of the system
(\ref{main_equations}):
\begin{equation}
\begin{split}
F_{2,3,0}=-3x^2_{3,0}{+}x_{3,0}x_{2,0}{+}2x_{2,0}x_{4,0}=0,\quad F_{3,4,0}={-}4x_{4,0}^2{+}3x_{4,0}x_{5,0}{+}3x_{3,0}x_{5,0}=0.\\
F_{2,4,0}=6x^2_{4,0}-4x_{3,0}x_{4,0}-x_{4,0}x_{5,0}+2x_{2,0}x_{5,0}-x_{3,0}x_{5,0}=0,\\
\end{split}
\end{equation}
Khakimdjanov in \cite{Kh} remarked that an affine variety, that is
determined by them in four-dimensional space with the coordinates
 $(x_{2,0},x_{3,0},x_{4,0},x_{5,0})$
is the union of three straight lines
$$
l_1=\{(t,0,0,0)\},
l_2=\left\{\left(t,\frac{t}{10},\frac{t}{70},\frac{t}{420}\right)\right\},
l_3=\{(0,0,0,t)\}.
$$
this observation allowed it to classify graded filiform Lie
algebras with one-dimensional homogeneous components in dimensions
 $n \ge 12$. although its initial list \cite{Kh} was not complete,
 in its later version \cite{Kh2} was also an algebra was missed, that was finally added in \cite{Mill1},
where the alternative classification of such algebras was carried
out. Let us note that non-zero points only of two first straight
lines continue to a solution $\Psi_0$ of entire system. Points of
the first straight line correspond to deformations with
$\Psi_0=t\Psi_{2,0}$ (they are isomorphic to ${\mathfrak m}_2$)
and the points of the second one correspond  to the cocycle
$\Psi_0=t\sum_{k=2}^{+\infty}6\frac{(k-2)!(k-1)!}{(2k-1)!}\Psi_{k,0}$
(they are isomorphic to $L_1$) and finally to the trivial solution
corresponds $\Psi_0=0$ and hence the algebra ${\mathfrak m}_0$.
this reasoning was repeated recently by  Wagemann and Fialowski
\cite{FialWag}.
\end{remark}
Let $\Psi=\Psi_0+\Psi_1+\Psi_2+\dots+\Psi_i+\dots$ be a
deformation of the algebra ${\mathfrak m}_0$ moreover
$\Psi_0=\Psi_1=\dots=\Psi_{s-1}=0$ and $\Psi_s \ne 0$ for some
integer $s \ge 0$.
\begin{definition}[\cite{Kh}] A cocycle
$\Psi_s$ is called {\it sill} cocycle  of deformation $\Psi$.
\end{definition}
A homogeneous cocycle $\Psi_s \in H^2_+({\mathfrak m}_0,{\mathfrak
m}_0)$ is sill cocycle for some deformation  $\Psi$ then and only
then $[\Psi_s,\Psi_s]=0$.
\begin{proposition}
Let $\Psi_s\in H^2_s({\mathfrak m}_0,{\mathfrak m}_0)$ be a sill
cocycle i.e. $[\Psi_s{,}\Psi_s]{=}0$. A Lie algebra ${\mathfrak
g}$ with the deformed bracket $[,]_{{\mathfrak m}_0}+\Psi_s$ has
the following structure of ${\mathbb N}$-graded Lie algebra with
trivial or one-dimensional homogeneous components:
$$
{\mathfrak g}=\bigoplus_{i=1}^{+\infty}{\mathfrak g}_i, \quad
\dim{\mathfrak g}_i=\left\{\begin{array}{l} 1, \; i=1, s{+}2, s{+}3,\dots; \\
0, \; 2\le i\le  s{+}1.
\end{array}\right.
$$
\end{proposition}
\begin{proof}
The grading is defined in a obvious way:
$$
{\mathfrak g}_1=\langle e_1 \rangle, \; {\mathfrak g}_i=\langle
e_{i{-}s} \rangle, i \ge s{+}2,$$ where $e_1,e_2,\dots,e_i,\dots$
is an infinite basis of the algebra ${\mathfrak m}_0$.
\end{proof}
The reverse is true also
\begin{proposition}
Let ${\mathfrak g}$ be an ${\mathbb N}$-graded Lie algebra such
that
$$
{\mathfrak g}=\bigoplus_{i=1}^{+\infty}{\mathfrak g}_i, \quad
[{\mathfrak g}_1,{\mathfrak g}_i] ={\mathfrak g}_{i{+}1}, i \ge
s{+}2, \quad
\dim{\mathfrak g}_i=\left\{\begin{array}{l} 1, \; i=1,  s{+}2, s{+}3,\dots; \\
0, \; 2\le i\le  s{+}1.
\end{array}\right.,
$$
for some integer $s \ge 0$. Then the Lie algebra ${\mathfrak g}$
is isomorphic to some deformation of ${\mathfrak m}_0$ defined by
a homogeneous cocycle $\Psi_s \in H^2_s({\mathfrak m}_0,{\mathfrak
m}_0)$ such that $[\Psi_s{,}\Psi_s]{=}0$.
\end{proposition}
\begin{proof}
The obvious verification.
\end{proof}
${\mathbb N}$-graded Lie algebras considered in two previous
propositions we will call  {\it ${\mathbb N}$-graded Lie algebras
with lacunas in the grading.  it is easy to construct examples of
this type from algebras already examined.
\begin{proposition}
Let ${\mathfrak g}=\bigoplus_{i=1}^{+\infty}{\mathfrak g}_i$ be an
${\mathbb N}$-graded Lie algebra then its subalgebra  ${\mathfrak
g}(s)={\mathfrak g}_1 \oplus
\bigoplus_{i=s{+}2}^{+\infty}{\mathfrak g}_i$ with $s \ge 1$ is
${\mathbb N}$-graded Lie algebra with lacunas in the grading from
$2$ to $s{+}1$.
\end{proposition}
The subalgebra $L_1(s)$ will be isomorphic to a semi-direct sum
${\mathbb K}\oplus L_{s{+}2}$. Obviously that${\mathfrak m}_0(s)
\cong {\mathfrak m}_2(s) \cong {\mathfrak m}_0$ with $s \ge 1$.
But not all algebras with lacunas can be obtained   by the
deletion of the basic vectors from some other algebra.
\begin{example}
A ${\mathbb N}$-graded Lie algebra ${\mathfrak m}_k$ is defined by
its infinite basis $e_1, e_k, e_{k{+}1}, \dots, e_i, \dots$ and
the structure relations:
$$
[e_1,e_i]=e_{i{+}1}, \; i \ge k,\quad
[e_k,e_i]=e_{k{+}i},\; i \ge k{+}1.
$$
The algebra  ${\mathfrak m}_k$ is isomorphic to the deformation of
${\mathfrak m}_0$ defined by the sill cocycle $\Psi_{2,k{-}2}$.
\end{example}
The question: how much exists the sill cocycles  $\Psi_s$ of the
grading equal to $s$? Let us try to find an answer in the case of
$s=2$. We will search for the solution of the system
(\ref{main_equations}) such that $x_{j,s}=0$ with $s \ne 2$. Let
us write down seven equations (\ref{main_equations}) on unknowns
$x_{2,2},x_{3,2},\dots, x_{8,2}$ moreover for simplification in
the record we will omitt the second upper script in the unknowns:
instead of $x_{j,2}$ we will write just $x_{j}$.
\begin{equation}
\label{syst_dve-dirki}
\begin{split}
F_{2,3,4}=5x_3^2-4x_2x_4+2x_2x_5-6x_3x_4+x_3x_5=0;\\
F_{2,4,4}=-15x_4^2+2x_2x_6-4x_2x_5+6x_3x_4+10x_4x_5-x_4x_6+3x_3x_5-x_3x_6=0;\\
F_{3,4,4}=6x_4^2-5x_3x_5+5x_3x_6+4x_4x_6-10x_4x_5=0;\\
F_{2,5,4}=35x_5^2-4x_2x_6+2x_2x_7+7x_3x_5+7x_3x_6-3x_3x_7-15x_5x_6+x_5x_7=0;\\
F_{3,5,4}=-21x_5^2-5x_3x_6+5x_3x_7+7x_4x_5+4x_4x_6-3x_4x_7+20x_5x_6-5x_5x_7=0;\\
F_{4,5,4}=7x_5^2-6x_4x_6+9x_4x_7-2x_4x_8-15x_5x_6+10x_5x_7-x_5x_8=0;\\
F_{3,6,4}=56x_6^2-5x_3x_7+5x_3x_8+9x_4x_7-8x_4x_8+8x_4x_6-36x_5x_6-35x_6x_7+6x_6x_8=0.\\
\end{split}
\end{equation}
After computing with $Maple$ we will obtain the following
collection of solutions of the system (\ref{syst_dve-dirki}):
\begin{equation}
\begin{split}
\{(t,0,0,0,0,0,0)\}, \{(0,0,0,0,0,u,t)\},
\{(0,0,0,0,t,4t,14t)\},\\
\left\{\left(\frac{t}{70},\frac{t}{420},\frac{t}{2310},\frac{t}{12012},
\frac{t}{60060},\frac{t}{291720},\frac{t}{1385670}\right)\right\}.
\end{split}
\end{equation}
One can show that only nontrivial points of the first and last
straight lines from this collection are extendable to a solution
of entire infinite system. The corresponding deformations are
isomorphic to algebras ${\mathfrak m}_4$ and $L_1(2)$. It follows
evidently from this proposition that up to an isomorphism there
are only three ${\mathbb N}$-graded Lie algebras of maximal class
with a lacuna in the grading equal to $2$: $L_1(1)$, ${\mathfrak
m}_3$ and ${\mathfrak m}_0$. The last statement was obtained
directly using computer calculations in \cite{FialWag}. In the
light of the Theorem  \ref{fialZ} and non published results of
computer calculations by Vaughan-Lee (see \cite{ShZ2} page 956),
the following conjecture is completely plausible
\begin{conjecture} Let ${\mathfrak g}$ be a ${\mathbb
N}$-graded Lie algebra such that
$$ {\mathfrak
g}=\bigoplus_{i=1}^{+\infty}{\mathfrak g}_i, \quad [{\mathfrak
g}_1,{\mathfrak g}_i] ={\mathfrak g}_{i{+}1}, i \ge k,\quad
\dim{\mathfrak g}_i=\left\{\begin{array}{l} 1, \; i=1, k, k{+}1,\dots; \\
0, \; 2\le i\le  k{-}1.
\end{array}\right..
$$
Then ${\mathfrak g}$ is isomorphic to one from three given
algebras
$$
{\mathfrak m}_0, \quad L_1(k{-}2), \quad {\mathfrak m}_k.
$$
\end{conjecture}

\section{The finite dimensional case}

In this section we will consider the affine variety  $M_{\it
Fil}(n)$ of $n$-dimensional filiform Lie algebras. One has
immediately to note that the structure of the system equations
that define $M_{\it Fil}(n)$ will depend on parity of $n$.

We will need  a finite dimensional version of the Theorem
\ref{main}:
\begin{theorem}
1) as the basis of the graded space
$H^2_+(\mathfrak{m}_0(n),\mathfrak{m}_0(n))=\oplus_{s \ge
0}^{n{-}5} H^2_s(\mathfrak{m}_0(n),\mathfrak{m}_0(n))$ one can
take the following collection of cocycles:
\begin{equation} \begin{split}
\Psi_{j,s}=\sum_{k=0}^{n{-}2j{-}1{-}s}e_{2j{+}1{+}s{+}k} \otimes
D_{-1}^k\omega(e^j {\wedge}e^{j{+}1}), \; 2\le j, \;
s \ge 0; \; 2j{+}1{+}s\le n.\\
\end{split}
\end{equation}

2) The following collection of cocycles
\begin{equation}
\begin{split}
\Psi_{i,j,s}=\sum_{k=0}^{n{-}i{-}2j{-}1{-}s}e_{i{+}2j{+}1{+}s{+}k}
\otimes D_{-1}^k\omega(e^i{\wedge}e^j {\wedge}e^{j{+}1}), \; 2\le
i <j, \; s \ge 0, \; i{+}2j{+}1{+}s\le n.
\end{split}
\end{equation}
can serve as a basis of the graded space
$H^3_+(\mathfrak{m}_0(n),\mathfrak{m}_0(n))=\oplus_{s \ge
0}^{n{-}9} H^3_s(\mathfrak{m}_0(n),\mathfrak{m}_0(n))$.
\end{theorem}
\begin{proof}
The first part of the theorem is just a reformulation of the
Vergne theorem on two-cohomology with the coefficients in the
adjoint representation of the algebra ${\mathfrak m}_0(n)$
\cite{V} (see also the Remark \ref{khakim}). Vergne have not
calculated the three-cohomology
$H^3(\mathfrak{m}_0(n),\mathfrak{m}_0(n))$. For the non-negative
part $H^3_+(\mathfrak{m}_0(n),\mathfrak{m}_0(n))$ the computation
by means of the spectral sequence \cite{Mill3} is still valid.
\end{proof}
We recall that according the Theorem  \ref{V_ne1} the following
holds on
\begin{proposition}
\label{neotr_def_fin_dim} An arbitrary filiform Lie algebra
 ${\mathfrak g}$ of odd dimension $n=2k{+}1$ is isomorphic to
 some non-negative deformation $\Psi=\Psi_0+\Psi_1+\Psi_2+\dots+\Psi_{n{-}5}$
of the graded Lie algebra ${\mathfrak m}_0(n)$. In the case of
even dimension $n=2k$ a filiform Lie algebra is non-negative
deformation of: either a) the graded Lie algebra ${\mathfrak
m}_0(n)$, either b) the graded Lie algebra ${\mathfrak m}_1(2k)$.

In each case the deformation  $\Psi$ can be chosen adapted:
\begin{equation}
\label{kogranica} \Psi(e_1, e_k)=0, \forall k \in {\mathbb N}.
\end{equation}
\end{proposition}

The Lie algebra ${\mathfrak m}_1(2k)$ can be regarded as
deformation ${\mathfrak m}_0(2k)$:
$$
[,]_{{\mathfrak m}_1(2k)}=[,]_{{\mathfrak m}_0(2k)}+\Psi_{k,{-}1},
$$
 $\Psi_{k,{-}1} \in H^2_{-1}({\mathfrak m}_0(2k),{\mathfrak
m}_0(2k))$ is defined by the formula
$$
\Psi_{k,{-}1}=e_{2k} \otimes \omega(e^k {\wedge}e^{k{+}1}).
$$
let us combine both cases into one, considering with $n=2k$
deformations $\Psi$ of the algebra ${\mathfrak m}_0(2k)$
$$
\Psi=x\Psi_{k,{-}1}+\Psi_0+\Psi_1+\Psi_2+\dots+\Psi_{n{-}5},
$$
where the variable $x$ can take only two values: $x=1$ (and we get
non-negative deformations of the algebra ${\mathfrak m}_1(2k)$) or
$x=0$ (we will deal with deformations of the algebra ${\mathfrak
m}_0(2k)$).

\begin{theorem}
\label{main_konechnomer} 1) Let $n=2k{+}1 \ge 9$. Let us consider
a linear combination
$$\Psi=\sum_{s=0}^{n{-}5}\sum_{j=2}^{\left[ \frac{n{-}1}{2}\right]}x_{j,s}\Psi_{j,s},$$
of basic cocycles  $\Psi_{j,s}$ from the subspace $\oplus_{s \ge
0}^{n{-}5} H^2_s(\mathfrak{m}_0(n),\mathfrak{m}_0(n))$. A cocycle
$\Psi$ satisfies the deformation equation if and only if its
coordinates $x_{j,s}$ satisfy the finite system of quadratic
equations
\begin{equation}
\label{sluchay_nechetnoy_dim} \left\{F_{j,q,r}=0, \quad 2 \le j <
q, \; 9 \le j{+}2q{+}1{+}r\le n, \; r \ge 0.\right.
\end{equation}
2) Let $n=2k \ge 10$. Let us consider a linear combination
$$\Psi=\sum_{s=0}^{n{-}5}\sum_{j=2}^{\left[ \frac{n{-}1}{2}\right]}x_{j,s}\Psi_{j,s}.$$
A cocyle $x\Psi_{k,{-}1}+\Psi$ satisfies the deformation equation
if and only if its coordinates $x_{j,s}$ with respect to the basis
$\Psi_{j,s}$ satisfy the following system of equations
\begin{equation}
\label{sluchay_chetnoy_dim}
\left\{
\begin{array}{ll}
F_{j,q,r}=0, & 2 \le j < q, \; 9 \le j{+}2q{+}1{+}r < 2k, \; r
\ge 0, \\
\tilde F_{j,q,r}=F_{j,q,r}+(-1)^{k{-}j{-}q}xG_{j,q,r}=0, & 2 \le j
< q,\; j{+}2q{+}1{+}r {=}2k, \; r
\ge 0, \\
\tilde F_{j,q,{-}1}=xG_{j,q,{-}1}=0, & 2 \le j < q, \; j{+}2q
{=}2k.
\end{array}\right.
\end{equation}
where the polynomials $F_{j,q,r}$ are defined by the formulas
(\ref{main_equations}) and the polynomial $G_{j,q,r}$ with $r \ge
-1$ is defined by the following
\begin{equation}
\begin{split}
G_{j,q,r}=\sum_{l{=}j}^{\left[\frac{j{+}q{-}1}{2}\right]}
({-}1)^{l} \binom{q{-}l{-}1}{l{-}j}x_{l,r{+}1}
+\sum_{l{=}j}^{\left[\frac{j{+}q}{2}\right]} ({-}1)^{l}
\binom{q{-}l}{l{-}j}x_{l,r{+}1} -(-1)^qx_{q,r{+}1}.
\end{split}
\end{equation}
\end{theorem}
\begin{proof}
The proof with $n=2k{+}1$ completely repeats our reasonings in
infinite dimensional case with the only the difference that the
Nijenhuis-Richardson bracket $[\Psi,\Psi]$ is decomposed in this
case with respect to the finite basis $H^3_+({\mathfrak
m}_0(n),{\mathfrak m}_0(n))$. Its elements are the cocycles
$\Psi_{j,q,r}$ with the upper scripts $j,q,r$, satisfying obvious
conditions $2 \le j < q, \; 9 \le j{+}2q{+}1{+}r \le n, \; r \ge
0$.

In the case of even dimension $n=2k$ one has to decompose the
square $[x\Psi_{k,{-}1}{+}\Psi,x\Psi_{k,{-}1}{+}\Psi]$ with
respect to the basis $\Psi_{j,q,r}$ of the subspace
$\oplus_{i\ge{-}1} H^3_i({\mathfrak m}_0(2k),{\mathfrak
m}_0(2k))$. Let us denote the coefficients of this decomposition
by $\tilde F_{j,q,r}$. One can easily obtain the expressions for
them using the formulas (\ref{main_equations}), where it is
necessary to change the limits of change of the index $t$: from
${-}1$ to $r$ (instead of  $0 \le t \le r$). As $x_{j,{-}1}=0$
with $j \ne k$, then the final formulas can be obtained by easy
change of (\ref{main_equations}).
\end{proof}
What we can say about the behavior of the number of unknowns and
number of equations of the system that determines  $M_{\it
Fil}(n)$ with an encrease in $n$? The number of umknowns (the
dimension $\dim H^2_+({\mathfrak m}(n),{\mathfrak m}(n))$) was
computed by Vergne \cite{V} and it is eqaul to
$\frac{1}{4}(n{-}3)^2$ with odd $n$ and
$\frac{1}{4}(n{-}2)(n{-}4)$ with even $n$. The number of equations
in  (\ref{sluchay_nechetnoy_dim}) is equal to the dimension $\dim
H^3_+({\mathfrak m}(n),{\mathfrak m}(n))$.
\begin{proposition}
$$
\dim H^3_+({\mathfrak m}(n),{\mathfrak
m}(n))=\sum_{r=3}^{n{-}6}P_3(r),
$$
where $P_3(r)$ denotes the partition number of a natural number
$r$ exactly into three terms.
\end{proposition}
\begin{proof}
The number of basic cocycles $\Psi_{i,j,s}$ with a fixed value of
upper script  $s$ is equal to $\sum_{k=1}^{n{-}s}\dim
H^3_{k}({\mathfrak m}_0(n))$. Using the formula
(\ref{chislo_cocyclov}) for $\dim H^3_k({\mathfrak m}_0(n))$ we
get
$$
\dim H^3_+({\mathfrak m}(n),{\mathfrak m}(n))=\sum_{s=0}^{n{-}5}
\sum_{k=1}^{n{-}s} \dim H^3_{k}({\mathfrak
m}_0(n))=\sum_{s=0}^{n{-}5}
\sum_{k=1}^{n{-}s}\left(P_3(k{-}6){-}P_3(k{-}7)
\right)=\sum_{s=0}^{n{-}5} P_3(n{-}s{-}6).
$$
Replacing in the last sum upper script $s$ by $r=n{-}s{-}6$ and
taking into account that $P_3({-}1)=P_3(0)=P_3(1)=P_3(2)=0$ the
required equality is obtained.
\end{proof}
It is possible to prove in a similar way that $\dim
H^2_+({\mathfrak m}(n),{\mathfrak m}(n))=\sum_{r=2}^{n{-}3}P_2(r).
$

Using the known recursion relation
$P_3(m)=P_3(m{-}3){+}P_2(m{-}3){+}P_1(m{-}3)$ one can write down
explicit formulas for the sum $\sum_{k=3}^{n{-}6}P_3(k)$ as a
function on $n$ (see \cite{H}) moreover the answer will depend on
the residue class of $n$ modulo $6$. We will not write them down.
Let us say only that according to the formula for $P_3(m)$ from
\cite{H} the sequence $\sum_{k=3}^{n{-}6}P_3(k)$ has an asymptotic
behavior $\frac{1}{36}n^3$ with the large numbers $n$. Usually
with the classification of filiform Lie algebras of small
dimensions they obtained the equations by substitution into
$[\Psi,\Psi]$ of all triples of basic vectiors $e_i, e_j, e_k$
with $2\le i<j<k\le n$ that gave $\binom{n{-}1}{3}$ equations
(moreover the explicit formula of equations for an arbitrary $n$
have not been written down) with the asymptotic behavior
$\frac{1}{6}n^3$ with the large numbers $n$.

But the dimension $\dim H^3_+({\mathfrak m}(n),{\mathfrak m}(n))$
will be answer only with odd values of $n$. For even ones $n=2k$
the variable $x$ appears as well as the additional series of
equations $\tilde F_{i,j,{-}1}=0$. The number of equations in this
series is equal to $\dim H^3_{2k{+}1}({\mathfrak
m}_0(2k))=P_3(2k{-}5)-P_3(2k{-}6)$. Finally we have: with even
$n=2k$ the number of equation in the system is equal to
$$\sum_{r=3}^{n{-}7}P_3(r)+P_3(n{-}5).$$
A question about the minimization of the number of equations that
determine $M_{\it Fil}(n)$ appeared long ago. An algorithm for
finding the explicit formulas of such systems was introduced in
\cite{EcMN}. IT seems that this algorithm  gives as the answer the
system that was given explicitly in the present paper. At least
our example $M_{\it Fil}(12)$ coincides completely with the
analogous example from \cite{EcMN}. It is understood that we not
say that the collection of polynomials $F_{j,q,r}$ the collection
of polynomials forms the minimal system of generators of the
corresponding ideal. It appears that it is not true beginning with
certain dimension  $n_0$.
\begin{example}
It is easy to find several first values $P_3(k)$:
$$
P_3(3){=}P_3(4){=}1, P_3(5){=}2, P_3(6){=}3, P_3(7){=}4,
P_3(8){=}5, P_3(9){=}7, P_3(10){=}8, P_3(11){=}10.
$$
with their aid it is not difficult to compose the table
$$
\begin{tabular}{|c|c|c|c|c|c|c|c|c|c|c|}
\hline
&&&&&&&&&&\\[-10pt]
dimension $n$ & 9 & 10 & 11 & 12 & 13 & 14 & 15 & 16 & 17 & 18\\
&&&&&&&&&&\\[-10pt]
\hline
&&&&&&&&&&\\[-10pt]
number of equations & 1 & 3 & 4 & 8 & 11 & 18 & 23 & 33 & 41 & 55\\
[3pt] \hline
\end{tabular}
$$
it completely coincides with the results of the computer
calculations from \cite{EcMN}.
\end{example}
\begin{example}
The variety $M_{\it Fil}(12)$ of $12$-dimensional filiform Lie
algebras.
\begin{equation}
\begin{split}
\tilde F_{2,5,{-}1}=x(-2x_{2,0}{+}3x_{3,0}{-}x_{5,0})=0,\;
F_{2,3,0}={-}3x^2_{3,0}{+}x_{3,0}x_{2,0}{+}2x_{2,0}x_{4,0}=0,\\
F_{2,3,1}=-2x_{2,0}x_{4,1}{+}7x_{3,0}x_{3,1}{-}3x_{4,0}x_{3,1}{-}3x_{4,0}x_{2,1}{-}x_{3,0}x_{4,1}=0;\\
F_{2,3,2}{=}4x_{3,1}^2{-}3x_{2,1}x_{4,1}{-}3x_{3,1}x_{4,1}{+}x_{2,2}(2x_{5,0}{-}x_{4,0}){+}
x_{3,2}(8x_{3,0}{-}6x_{4,0}{+}x_{5,0}){-}x_{4,2}(x_{3,0}{+}2x_{2,0}){=}0,\\
\tilde F_{2,3,3}{=}x_{2,2}(2x_{5,1}{-}4x_{4,1}){+}x_{3,2}(9x_{3,1}{-}6x_{4,1}{+}x_{5,1}){-}x_{4,2}(3x_{3,1}{+}3x_{2,1}){+}\\
x_{2,3}(5x_{5,0}{-}5x_{4,0}){+}x_{3,3}(9x_{3,0}{-}10x_{4,0}{+}4x_{5,0}){+}x_{4,3}(x_{3,0}{-}2x_{2,0}){-}
x(2x_{2,4}+x_{3,4}){=}0,\\
F_{2,4,0}=6x^2_{4,0}{-}4x_{3,0}x_{4,0}{-}x_{4,0}x_{5,0}{+}2x_{2,0}x_{5,0}{-}x_{3,0}x_{5,0}{=}0,\\
\tilde
F_{2,4,1}=-2x_{2,0}x_{5,1}{+}5x_{3,0}x_{4,1}{+}x_{3,0}x_{5,1}{-}10x_{4,0}x_{4,1}{+}4x_{5,0}x_{4,1}{-}
3x_{5,0}x_{2,1}+\\
+4x_{4,0}x_{3,1}{+}2x_{5,0}x_{3,1}{-}6x_{4,0}x_{4,1}{+}x_{4,0}x_{5,1}{+}
x(2x_{2,2}-x_{3,2}-x_{4,2})=0;\\
\tilde
F_{3,4,0}=-4x_{4,0}^2{+}3x_{4,0}x_{5,0}{+}3x_{3,0}x_{5,0}{+}x(2x_{3,1}{+}x_{4,1})=0.
\end{split}
\end{equation}
\end{example}

\end{document}